\documentclass[11pt,reqno]{amsart}

\usepackage{enumerate,url,amssymb,  mathrsfs}
\usepackage{graphicx}
\usepackage{esint,comment}

\usepackage[pdftex,bookmarks,pdfnewwindow,plainpages=false,unicode,pdfencoding=auto]{hyperref}

 \hypersetup{
pdfauthor={Oleg Ivrii},
pdftitle={On Makarov's principle in conformal mapping},
pdfsubject={Conformal mapping},
bookmarksdepth={4}
}

\pdfoutput=1

\newtheorem{theorem}{Theorem}[section]
\newtheorem{lemma}[theorem]{Lemma}
\newtheorem*{lemma*}{Lemma}

\newtheorem{corollary}[theorem]{Corollary}

\theoremstyle{definition}

\theoremstyle{remark}
\newtheorem*{remark}{Remark}

\numberwithin{equation}{section}

\DeclareMathOperator{\var}{Var}

\DeclareMathOperator{\bottom}{bot}

\DeclareMathOperator{\re}{Re}
\DeclareMathOperator{\im}{Im}

\DeclareMathOperator{\LIL}{LIL}

\DeclareMathOperator{\bs}{\mathbf S}

\DeclareMathOperator*{\esssup}{ess\,sup}
\DeclareMathOperator{\Hdim}{H.dim }
\DeclareMathOperator{\Mdim}{M.dim }

\DeclareMathOperator{\ei}{I}

\DeclareMathOperator{\bel}{dil. }
\DeclareMathOperator{\bH}{\mathbf{H}}
\DeclareMathOperator{\bS}{\mathbf{S}}
\DeclareMathOperator{\bl}{B}
\DeclareMathOperator{\fuchs}{F}
\DeclareMathOperator{\good}{good}
\DeclareMathOperator{\bad}{bad}

\def\XXint#1#2#3{{\setbox0=\hbox{$#1{#2#3}{\int}$}
\vcenter{\hbox{$#2#3$}}\kern-.5\wd0}}

\def\le{\leqslant}
\def\ge{\geqslant}
\begin{document}
\baselineskip6mm
\vskip0.4cm
\title{On Makarov's principle in conformal mapping}

\author[O. Ivrii]{Oleg Ivrii}
\address{Department of Mathematics and Statistics, University of Helsinki,
         P.O. Box 68, FIN-00014, Helsinki, Finland}
\email{oleg.ivrii@helsinki.fi}



\thanks{
The author was supported by the
Academy of Finland, project no.~271983.}

\begin{abstract}
We examine several characteristics of conformal maps that resemble the variance
of a Gaussian: asymptotic variance, the constant in Makarov's law of iterated logarithm and the second derivative of the integral means spectrum at the origin.
While these quantities need not be equal in general, they agree for domains whose boundaries are regular fractals such as Julia sets or limit sets of quasi-Fuchsian groups. We give a new proof of these dynamical equalities.
We also show that these characteristics have the same universal bounds and prove a central limit theorem for extremals.
Our method is based on analyzing the local variance of dyadic martingales associated to Bloch functions.
\end{abstract}

\maketitle

\section{Introduction}\label{se:introduction}

Let $\Omega \subset \mathbb{C}$ be a simply connected domain in the plane whose boundary is a Jordan curve.
In 1985, N.~Makarov \cite{makarov85} introduced probabilistic techniques into the theory of conformal mapping to show that the harmonic measure on 
$\partial \Omega$ necessarily has Hausdorff dimension 1.
This is quite surprising for domains $\Omega$ with $\Hdim \partial \Omega > 1$: for such domains, Makarov's theorem suggests that Brownian motion started at an interior point $z_0 \in \Omega$ only hits a small subset of the boundary. In probabilistic terms, the above result may be viewed as analogue of the {\em law of large numbers} for random variables.

In order to obtain finer information about the metric properties of the harmonic measure, Makarov proved a {\em law of iterated logarithm} for Bloch functions.
Loosely speaking, Makarov's work suggests that conformal maps (at least to nice fractal domains) resemble Gaussians:
$$
\mathcal N_{\mu, \sigma^2}(t) = \frac{1}{\sigma \sqrt{2\pi}} \int_{-\infty}^t \exp \biggl (- \frac{(x-\mu)^2}{2\sigma^2} \biggr ) dx.
$$
A Gaussian is specified by two parameters: its mean $\mu$ and variance $\sigma^2$. The variance $\sigma^2$ may be extracted in
several ways: through the {\em central limit theorem} (CLT), the {\em law of the iterated logarithm} (LIL) or via 
{\em exponential integrability} estimates. These notions lead to several different characteristics of conformal maps $f: \mathbb{D} \to \mathbb{C}$. 
It is easier to define these characteristics in terms of the associated Bloch functions  $b_f := \log f'.$ As is well-known, each function $b_f$ arising in this way,  satisfies a bound of the form $$\|b_f\|_\mathcal B := \sup_{z \in \mathbb{D}} |f'(z)|(1-|z|^2) \le 6.$$
We recall the definitions:
\begin{itemize}
\item The {\em asymptotic variance}
\begin{equation}
\label{eq:av}
\sigma^2(b) = \limsup_{r\to1} \frac{1}{2\pi |\log(1-r)|} \int_{|z|=r} |b(z)|^2 \, |dz|.
\end{equation}
\item The {\em LIL constant}
\begin{equation}
\label{eq:lil}
C_{\LIL}(b) =  \esssup_{\theta \in [0,2\pi)} \ \Biggl\{
\limsup_{r \to 1} \frac{ |b(re^{i\theta})|}{\sqrt{\log \frac{1}{1-r} \log \log\log  \frac{1}{1-r}}} \Biggr\}.
\end{equation}
\item The {\em integral means spectrum}
\begin{equation}
\beta_b(\tau) = \limsup_{r \to 1} \frac{1}{ |\log(1-r)|} \cdot \log \int_{|z|=r} \bigl |e^{\tau b(z)} \bigr | \, |dz|, \qquad \tau \in \mathbb{C}.
\end{equation}
\end{itemize}

As hinted above, in dynamical situations, the above characteristics are linked by an explicit relation:

\begin{theorem} 
\label{dynamical-connections}
Suppose $f(z)$ is a conformal map, such that the image of the unit circle $f(\mathbb{S}^1)$ is a Jordan curve, invariant under a hyperbolic conformal dynamical system.
Then, 
\begin{equation}\label{eq:1234}
2 \frac{d^2}{d\tau^2}\biggl |_{\tau=0} \beta_{\log f'}(\tau) = \sigma^2(\log f') = C^2_{\LIL}(\log f').
\end{equation}
\end{theorem}

The equalities in (\ref{eq:1234}) are mediated by a fourth quantity 
involving the {\em dynamical asymptotic variance} of a H\"older continuous potential from thermodynamic formalism. Theorem 
\ref{dynamical-connections} has a rich history: 
the connection with $C^2_{\LIL}$ is due to Przytycki, Urba\'nski, Zdunik  \cite{PUZ2}, with integral
means due to Makarov and Binder \cite{makarov99, binder}, and with $\sigma^2$ by McMullen \cite{mcmullen}, see also \cite{AIPP} for additional details.
One of our central objectives is to give a new proof of Theorem \ref{dynamical-connections} that does not involve thermodynamic formalism.
Instead, we work with a new central quantity: the {\em local variance} of a dyadic martingale associated to a Bloch function.
The definition will be given in Section \ref{sec:probability}.

We emphasize that the above quantities are unrelated in general. We refer the reader to \cite{BaMo, le-zinsmeister} for a discussion and interesting
examples.
Nevertheless, one can ask if  the above characteristics agree on the level of universal bounds, taken over all conformal maps.
We show that this is essentially the case; however, in order to be able to localize these characteristics, we are forced to restrict to conformal maps that have quasiconformal extensions with bounded distortion.

To be concrete, let $\bs$ be the class of conformal maps $f: \mathbb{D} \to \mathbb{C}$ normalized so that $f(0) = 0$ and $f'(0) = 1$, and for $0 < k < 1$, let $\bs_k \subset \bs$ denote the collection of maps that admit a $k$-quasiconformal extension to the complex plane.
Let $B_k(\tau) := \sup_{f \in \bs_k} \beta_{\log f'}(\tau)$.

\begin{theorem}
\label{main-thm}
$$
\lim_{\tau \to 0} \frac{B_k(\tau)}{|\tau|^2/4} =  \sup_{f \in \bs_k} \sigma^2(\log f') = \sup_{f \in \bs_k} C_{\LIL}^2(\log f').
$$
\end{theorem}

The above quantity will be denoted $\Sigma^2(k)$. As discussed in \cite{qcdim}, $\Sigma^2(k)/k^2$ is a non-decreasing convex function of $k$. It is currently known that $$0.93 < \lim_{k \to 1^-} \Sigma^2(k) < (1.24)^2.$$ We refer the reader to \cite[Section 8]{AIPP} for the lower bound
and to \cite{HK, hedenmalm-shimorin} for the upper bound.
A theorem of Makarov \cite{makarov87}, 
 \cite[Theorem VIII.2.1]{GM} shows:
\begin{corollary}
\label{main-cor2}
{\em (i)} Let $\Omega = f(\mathbb{D})$ be the image of the unit disk and $z_0$ be a point in $\Omega$.
The harmonic measure $\omega_{z_0}$ on $\partial \Omega$, as viewed from $z_0$, is absolutely continuous with respect to the Hausdorff measure  $\Lambda_{h(t)}$,
$$h(t)=t\,\exp\left\{C\sqrt{\log\frac{1}{t}\log\log
\log\frac{1}{t}}\right\},\qquad 0<t<10^{-7},$$
for any $C \ge C_{\LIL}(b_f)$. In particular, $C = \sqrt{\Sigma^2(k)}$ works.

{\em (ii)} Conversely, if $C < \sqrt{\Sigma^2(k)}$, there exists a conformal map $f \in \bs_k$ for which $\omega_{z_0} \perp \Lambda_{h(t)}$.
\end{corollary} 

The connections to LIL in Theorem \ref{main-thm} and Corollary \ref{main-cor2} were originally proved together with I.~Kayumov using a different method than 
presented here.

\newpage

We now turn to the infinitesimal analogues of the above results from the point of view of universal Teichm\"uller space. Consider the quantity
\begin{equation}
\label{eq:sigma2-def}
\Sigma^2 := \sup_{|\mu| \le \chi_{\mathbb{D}}} \sigma^2(\mathcal S\mu),
\end{equation}
where
\begin{equation}
\label{eq:beurling-def}
\mathcal S\mu(z)=-\frac{1}{\pi} \int_{\mathbb{D}}\frac{\mu(w)}{(z-w)^2} \, |dw|^2, \qquad |z| > 1,
\end{equation}
is the {\em Beurling transform.}
A simple computation shows that $\mathcal S \mu \in \mathcal B(\mathbb{D}^*)$, the Bloch space of the exterior unit disk.
Furthermore, if $w^{t\mu}$ is the principal solution to the Beltrami equation $\overline{\partial}w = t \mu \, \partial w$, then
\begin{equation}
\label{eq:quadratic-error}
\| \log (w^{t\mu})' - \mathcal S\mu \|_{\mathcal B(\mathbb{D}^*)} = \mathcal O(|t|^2), \qquad t \in \mathbb{D},
\end{equation}
 for instance, see
\cite[Section 2]{qcdim}. In particular, $\Sigma^2 = \lim_{k \to 0} \Sigma^2(k)/k^2$.

 In order to keep the discussion in the disk,
it is sometimes preferable to work with the {\em Bergman projection}
\begin{equation}
\label{eq:bergman-def}
 P\mu(z)=\frac{1}{\pi} \int_{\mathbb{D}}\frac{\mu(w)}{(1-z\overline{w})^2} \, |dw|^2.
\end{equation}
The two operators are connected by $\mathcal S\mu(z) = -(1/\overline{z}^2) P\mu(1/\overline{z})$. 

The infinitesimal analogue of Theorem \ref{main-thm} can be expressed in terms of the  {\em rescaled integral means spectrum}
\begin{equation}
\label{eq:inf-ims}
B_0(\tau) \, := \, \lim_{k \to 0}  B_k(\tau/k) \, = \, 
 \sup_{|\mu| \le \chi_{\mathbb{D}}} \beta_{P\mu} (\tau).
 \end{equation}
\begin{remark}
 Since the collection of Bloch functions $\{P\mu, |\mu| \le \chi_{\mathbb{D}}\}$ is invariant under rotation by $e^{i\theta} \in \mathbb{S}^1$,
 $B_0(\tau)$ only depends on $|\tau|$.
\end{remark}

\begin{corollary}
\label{main-thm-inf}
$$
\Sigma^2 \, := \, \lim_{\tau \to 0} \frac{B_0(\tau)}{|\tau|^2/4}  \, = \, \sup_{|\mu| \le \chi_\mathbb{D}} \sigma^2(P\mu) \, = \, \sup_{|\mu| \le \chi_\mathbb{D}} C_{\LIL}^2(P\mu).
$$
\end{corollary}

  The quantity $\Sigma^2$ was first studied in \cite{AIPP}, where it was established that
$
0.87913 \le \Sigma^2 \le 1,
$
while Hedenmalm \cite{hedenmalm} proved the strict inequality $\Sigma^2 < 1$.
In \cite{AIPP}, the original motivation for investigating $\Sigma^2$ arose from the connection between dimensions of quasicircles and McMullen's identity (Theorem
\ref{dynamical-connections}).
 Let $D(k)$ denote the maximal Minkowski dimension of a $k$-quasicircle, the image of the unit circle under a $k$-quasiconformal mapping of the plane.
 
As is well-known, the problem of finding $D(k)$ reduces to the study of integral means 
via the anti-symmetrization procedure of \cite{kuhnau, smirnov} and the relation \cite[Corollary 10.18]{Pomm}
\begin{equation}
\label{eq:beta-dimension}
\beta_{f}(t) = t-1 \quad \Longleftrightarrow \quad t = \Mdim f(\mathbb{S}^1), \qquad f \in \bS_k.
\end{equation}
In \cite{qcdim}, the author modified the argument of Becker and Pommerenke \cite{BP} for estimating integral means to show the asymptotic expansion
\begin{equation}
\label{eq:qcdim-exp}
D(k) = 1 + k^2\Sigma^2 + \mathcal O(k^{8/3-\varepsilon}).
\end{equation}
Together with Hedenmalm's estimate, this improves on Smirnov's bound $D(k) \le 1+k^2$ from \cite{smirnov}.
In Section \ref{sec:ims}, we will give an estimate for $B_0(\tau)$ which implies (\ref{eq:qcdim-exp}), albeit with
a slightly weaker error term.

An a priori difficulty in studying $\Sigma^2$ is that the extremal problem (\ref{eq:sigma2-def}) has infinitely many solutions. For instance, one can take an extremal $\mu$
and modify it in an arbitrary manner on a compact subset of the disk. Alternatively, one can pullback an extremal $\mu(z) \frac{d\overline{z}}{dz}$ by a Blaschke product
$B: \mathbb{D} \to \mathbb{D}$. Further, given two extremals $\mu, \nu$, one can glue them together $\mu \cdot \{\chi_{\re z < 0}\} + \nu \cdot \{\chi_{\re z > 0}\}$ to
form yet another extremal.
In \cite[Section 6]{AIPP}, extremals were studied indirectly via fractal approximation:
\label{AIPP-fat}
\begin{equation}
\label{eq:AIPP-fat}
\Sigma^2 = \sup_{\mu \in M_{\ei}, \ |\mu| \le \chi_{\mathbb{D}}} \sigma^2(\mathcal S\mu),
\end{equation}
where  $M_{\ei}$ is the class of Beltrami coefficients that are {\em eventually-invariant} under $z \to z^d$ for some $d \ge 2$, i.e.~satisfying
$(z^d)^*\mu = \mu$ in some open neighbourhood of the unit circle. In particular, since Theorem \ref{dynamical-connections} is applicable
to conformal maps $w^\mu(z)$, $z \in \mathbb{D}^*$, with $\mu \in M_{\ei}$, $\|\mu\|_\infty < 1$, the inequality $\Sigma^2 \le 1$ follows from Smirnov's $1+k^2$ bound.

In Section \ref{sec:clt} of the present paper, we show that extremal Bloch functions $b = P\mu$ obey a central limit theorem. For a fixed $r <1$,
we may consider
\begin{equation}
\label{eq:rescaled-bloch}
 \tilde b_r(\theta) := \frac{b(re^{i \theta})}{\sqrt {|\log(1-r)|}}
\end{equation}
as a random variable with respect to the probability measure $|dz|/2\pi$.
\begin{theorem}
\label{complex-gaussian}
Suppose $|\mu| \le \chi_{\mathbb{D}}$.
Given $\varepsilon > 0$, there exists a $\delta > 0$ such that if $r$ is sufficiently close to 1 and
 $$
 \sigma^2(P\mu,r) \, = \, \frac{1}{2\pi |\log(1-r)|} \int_{|z|=r} |P\mu(z)|^2 \, |dz| \, > \, \Sigma^2 - \delta,
 $$ 
then the distribution of $(\widetilde{P\mu})_r$ is close to a complex Gaussian, of mean 0 and variance $\Sigma^2$,
 up to an additive error of at most $\varepsilon$.
In other words, $\re (\widetilde{P\mu})_r$ and $\im (\widetilde{P\mu})_r$ approximate independent real Gaussians of variance $\Sigma^2/2$.
\end{theorem}
Summarizing, the above theorem says that in the problem of maximizing asymptotic variance, all extremals are Gaussians. In particular, 
{\em extremality invokes fractal structure}.

\subsection*{Notation} Let $\rho_*(z) = \frac{2}{|z|^2-1}$ be the density of the hyperbolic metric in the exterior unit disk $\mathbb{D}^*$ and 
$\rho_{\mathbb{H}}(z)= 1/y$ be the corresponding density in the upper half-plane.
To compare quantities, we use $A \gtrsim B$ to denote $A > C \cdot B$ for some constant $C > 0$. The notation $\fint f(t) \, g(t)dt$ 
 denotes the average value of the function $f(t)$ with respect to the measure $g(t)dt$.

 \section{Background in probability}
\label{sec:probability}

In this section, we discuss martingale analogues of the characteristics of conformal maps mentioned in the introduction. We show  that they
are controlled by the local variation.

\subsection{Martingales and square functions} 
Let $p \ge 2$ be an integer, $\mathcal D_k$ be the collection of 
$p$-adic intervals $\bigl [j \cdot p^{-k}, (j+1) \cdot p^{-k} \bigr ]$ contained in $[0,1]$, and
$\mathcal M_k$ be the $\sigma$-algebra generated by $\mathcal D_k$.
A (complex-valued) {\em $p$-adic martingale} $X$ on $[0,1]$ is
a sequence of functions $\{X_k\}_{k=0}^\infty$ such that 
\begin{itemize}
\item[(i)] $X_k$ is measurable with respect to $\mathcal M_k$,
\item[(ii)]  $\mathbb{E}(X_k | \mathcal M_{k-1}) = X_{k-1}$.
\end{itemize}
We typically view $X$ as a function from $\bigcup_{k = 0}^\infty \mathcal D_k$ to the complex numbers which satisfies the
{\em averaging property}
\begin{equation}
\label{eq:martingale-condition}
X_I = \frac{1}{p} \sum_{i=1}^{p} X_{I_i},
\end{equation}
where the sum ranges over the $p$-adic children of $I$.
For a point $x \in [0,1]$, let $I_j(x) \in \mathcal D_j$ denote the $p$-adic interval of length $p^{-j}$ containing $x$, and
$\Delta_j(x) = X_{I_j(x)}  - X_{I_{j-1}}(x)$
be the {\em jump} at step $j$.
The {\em $p$-adic square function} is given by
\begin{equation}
\langle X \rangle_n := \sum_{j=1}^n |\Delta_j(x)|^2.
\end{equation}
We say that a martingale has {\em bounded increments} if $|\Delta_j(x)| < C$ for all $x \in [0,1]$ and $j \ge 1$. For such martingales, we define:
\begin{itemize}
\item The {\em asymptotic variance}
\begin{align*}
\sigma^2(X) & = \limsup_{n \to \infty} \frac{1}{n} \int_0^1 |X_n(x)|^2 \, dx \\
& = \limsup_{n \to \infty} \frac{1}{n} \int_0^1 \langle X \rangle_n \, dx.
\end{align*}
(The equality follows from the orthogonality of the jumps.)

\item The {\em LIL constant}
$$
C_{\LIL}(X) = \esssup_{x \in [0,1]} \ \biggl\{ \limsup_{n \to \infty} \frac{|X_n(x)|}{\sqrt{ n \log\log n}}\biggr\}.
$$
\item The {\em integral means spectrum}
$$
\beta_X(\tau) = \limsup_{n \to \infty} \frac{1}{n} \cdot \log \int_0^1 \bigl | e^{\tau X_n(x)}\bigr | \, dx, \qquad \tau \in \mathbb{C}.
$$

\end{itemize}

\subsection{Local variance} For a $p$-adic interval $I$, we define the {\em local variance} of $X$ at $I$ as
\begin{equation}
\var_I X= \frac{1}{p} \sum_{i=1}^{p} |X_{I_i} - X_I|^2.
\end{equation}
More generally, we can consider
\begin{equation}
\var_I^n X= \frac{1}{n} \biggl [ \frac{1}{p^n} \sum_{i=1}^{p^n} |X_{I_i} - X_I|^2\biggr],
\end{equation}
where we sum over all $p$-adic grandchildren of $I$ of length $p^{-n}|I|$. Polarizing, we obtain  
 the notion of {\em local covariance}  \begin{equation}
\var_I^n (X, Y) = \frac{1}{n} \biggl [ \frac{1}{p^n} \sum_{i=1}^{p^n} (X_{I_i} - X_I) \overline{(Y_{I_i} - Y_I)} \biggr]
\end{equation}
of two $p$-adic martingales $X$ and $Y$.
Our aim is to show that the local variance controls the above characteristics:
\begin{theorem}
\label{local-variance}
Suppose $S$ is a real-valued martingale with bounded increments. Let $m = \inf_I \var_I$ and $M = \sup_I \var_I$. Then, 

{\em (i)} For a.e.~$x \in [0,1]$, 
$$
m \, \le \, \liminf \frac{\langle S \rangle_n}{n} \, \le \, \limsup \frac{\langle S \rangle_n}{n} \, \le \, M,
$$

{\em (ii)} $m \le \sigma^2(S) \le M$,

{\em (iii)} $m \le (1/2) \cdot C_{\LIL}^2(S) \le M$,

{\em (iv)} For $t \in \mathbb{R}$,
$$m \, \le \, \liminf_{t \to 0} \frac{\beta_S(t)}{t^2/2} \, \le \, \limsup_{t \to 0} \frac{\beta_S(t)}{t^2/2} \le M.$$
\end{theorem}

To evaluate the LIL constant of a martingale, we use a result of W.~Stout \cite{stout}, which is stated explicitly in the form below in \cite[Theorem 2.6]{makarov90}:
\begin{lemma}[Stout]
\label{stout-thm}
 If $S_n$ is a real-valued martingale with bounded increments, then
\begin{equation}
\label{eq:stout-thm}
\limsup_{n \to \infty} \frac{|S_n(x)|}{\sqrt{2 \langle S \rangle_n \log\log \langle S\rangle_n}} = 1,
\end{equation}
almost surely on the set $\{x : \langle S \rangle_n = \infty\}$.
\end{lemma}

\begin{proof}[Proof of Theorem \ref{local-variance}] (i) Consider the auxiliary martingale $T$ with $T_{[0,1]}=0$ and jumps 
$$
T(I_i) -T(I) \, := \, |X(I_i) - X(I)|^2 - \frac{1}{p} \sum_{j=1}^{p}  |X(I_j) - X(I)|^2,
$$
where $I_1, I_2, \dots, I_{p}$ are the $p$-adic children of $I$.
Note that $T$ has bounded increments since $X$ does. Applying Lemma \ref{stout-thm} to the martingale $T$, we see that
 $$T_n(x) = \mathcal O \bigl (\sqrt{n \log \log n} \bigr ) = o(n), 
\qquad \text{for a.e~}x \in [0,1].
$$
In particular,
\begin{equation} 
 \frac{\langle X\rangle_n}{n} \, = \,
\frac{1}{n} \sum_{k=1}^n  \var_{I_k(x)} + o(1).
\end{equation}
The rest is easy: (ii) is trivial, (iii) follows from (i) by Stout's lemma, while (iv) follows from the expansion
$$
\frac{1}{p} \sum_{j=1}^{p} \exp(t \Delta_j) = 1 + \frac{t^2}{2} \biggl ( \frac{1}{p} \sum_{j=1}^{p} \Delta_j^2 \biggr) + \mathcal O(t^3).
$$
This proof is complete.
\end{proof}
The same argument shows:
\begin{lemma}
\label{same-variance}
If two real-valued martingales $S_1, S_2$ satisfy $\var_I S_1 = \var_I S_2$ for all $I$, then they have the same
LIL constant.
 More generally, $$(1/2)\, |C_{\LIL}^2(S_1) - C_{\LIL}^2(S_2)| \le \sup_I |\var_I S_1 - \var_I  S_2|.$$
\end{lemma}
The above lemma also holds for the other 
characteristics discussed in Theorem \ref{local-variance}.

\subsection{Some useful facts}
For future reference, we record two martingale estimates. Assume for simplicity that $S$ is a real-valued dyadic martingale with $S_{[0,1]} = 0$ and
$|\Delta_j(x)| \le 1$. The {\em sub-Gaussian estimate} says that
\begin{equation}
\label{eq:subgaussian}
\bigl | \{x \in [0,1] : |S_n| > t \} \bigr | \le e^{-c t^2/n}.
\end{equation}
for some $c > 0$. The sub-Gaussian estimate is a consequence of a more general statement, see \cite[Proposition 2.7]{makarov90}. 
Another proof is given in \cite{hedenmalm2}.
Integrating (\ref{eq:subgaussian}), we obtain bounds for the moments
\begin{equation}
\label{eq:moments}
\frac{1}{\Gamma(p+1)} \int_0^1 |S_n|^{2p} \, dx \le (Cn)^p, \qquad p \ge 0.
\end{equation}

\section{Bloch martingales}

In this section, we review Makarov's construction of the dyadic martingale associated to a Bloch function. We then give an approximate formula
for the local variance. 
For convenience, we work in the upper half-plane where the computations are slightly simpler.
Therefore, let us imagine that $b$ is a holomorphic function on $\mathbb{H}$ with 
\begin{equation}
\label{eq:def-bloch}
\|b\|_{\mathcal B(\mathbb{H})} = (1/2) \, \sup_{z \in \mathbb{H}}\, y \cdot |b'(z)| \le 1.
\end{equation}
Here, we assume that $b$ lies in the Bloch unit ball in order to not have to write the Bloch norm all the time.

A dyadic interval $I = [x_1, x_2] \subset [0,1]$ defines a {\em 1-box}
$$
\square_I \, = \, \bigl \{ w : \re w \in [x_1, x_2], \, \im w \in [(x_2 - x_1)/2, \, x_2 - x_1] \bigr \}
$$
in the upper half-plane. The {\em $n$-box} $\square_I^n$ is defined as the union of 1-boxes associated to $I$ and to all dyadic intervals
 contained in $I$ of length at least $2^{-n+1}|I|$. For instance,
 $$
\square_{[0,1]}^n \, = \, \bigl \{ w : \re w \in [0, 1], \, \im w \in [2^{-n}, 1] \bigr \}.
$$
We use $z_I = (x_1+x_2)/2 + (x_2 - x_1)i$ to denote the midpoint of the top edge of $\square_I$.
Following Makarov \cite{makarov90}, let  $B$ be the complex-valued dyadic martingale given by 
\begin{equation}
B_I = \lim_{y \to 0^+} \int_I b(x+iy) dx.
\end{equation}
Makarov showed that the above limit exists and satisfies
\begin{equation}
|b(z_I) - B_I | = \mathcal O(1).
\end{equation}
In particular, 
\begin{equation}
\label{eq:bloch-martingale-condition}
|B_I - B_J| \le C,
\end{equation} whenever $I, J$ are adjacent dyadic
intervals of the same size. This is stronger than simply saying that $B$ has bounded increments because $I, J$ may have different parents.
Makarov \cite{makarov90} observed that the converse also holds: if a dyadic martingale on $[0,1]$ satisfies the above property, it comes from some Bloch function $b(z)$. It is therefore natural to refer to martingales satisfying (\ref{eq:bloch-martingale-condition}) as  {\em Bloch martingales}.

 One may view dyadic martingales abstractly, defined on the dyadic tree. The notion of a Bloch martingale; however, requires an identification of the dyadic tree with $[0,1]$. One useful fact to keep in mind is:
 
\begin{lemma}[Transmutation principle]
For an abstract dyadic martingale with bounded increments, there is an embedding to $[0,1]$ so that it is Bloch.
\end{lemma}

The proof is simple and we leave it as an exercise to the reader.

\begin{theorem}
 \label{variance-thm}
 Suppose $I \subset [0,1]$ is a dyadic interval and $I_1, I_2, \dots, I_{2^n}$ are its dyadic grandchildren of length $2^{-n}|I|$. Then,
\begin{equation}
\label{eq:ltg1}
\frac{1}{\log 2} \cdot \var_I^n B = \fint_{\square_I^n} \, \biggl | \frac{2b'}{\rho_{\mathbb{H}}}(z) \biggr |^2 \, \frac{|dz|^2}{y} +
 \mathcal O \bigl (\|b\|^2_{\mathcal B}/\sqrt{n} \bigr ).
\end{equation}
Furthermore, we have the ``complexification'' relations
\begin{equation}
\label{eq:ltg2}
\var_I^n (\re B) = \frac{1}{2} \cdot \var_I^n B +  \mathcal O \bigl (\|b\|^2_{\mathcal B}/\sqrt{n} \bigr)
\end{equation}
and
\begin{equation}
\label{eq:ltg3}
\var_I^n (\re B,\,\im B) =  \mathcal O \bigl (\|b\|^2_{\mathcal B}/\sqrt{n} \bigr ).
\end{equation}
\end{theorem}
\begin{proof}
Due to scale invariance, we only need to consider the case when $I = [0,1]$ and $\|b\|_{\mathcal B(\mathbb{H})} = 1$.
Since $B$ has bounded jumps, 
\begin{equation*}
\langle B \rangle_n \lesssim n \qquad \text{and} \qquad
\frac{1}{2^n} \sum_{j=1}^{2^n} |B_{I_j} - B_I|^2  \lesssim n.
\end{equation*}
The Cauchy-Schwarz inequality gives
\begin{equation}
\frac{1}{2^n} \sum_{j=1}^{2^n} |B_{I_j} - B_I|  \lesssim \sqrt{n}.
\end{equation}
For $z \in I_j$, the Bloch property $|b(z)-b(z_{I_j})| = \mathcal O(1)$ implies
\begin{equation}
\label{eq:l1norm}
\int_{\bottom( \square_I^n)} |b(z)-b(z_I)| \, ds \lesssim \sqrt{n},
\end{equation}
where we integrate over the bottom side of $\square_I^n$.
Similarly,
\begin{equation}
\label{eq:sqf-appx}
\biggl |  \int_{\bottom( \square_I^n)} |b(z)-b(z_I)|^2 \, ds  -  \frac{1}{2^n} \sum_{j=1}^{2^n} |B_{I_j} - B_I|^2
\biggr |
 \, \lesssim \, \sqrt{n}.
\end{equation}
Following \cite{nicolau}, we apply Green's identity
$$
\int_\Omega (u \Delta v - v \Delta u) dxdy = \int_{\partial \Omega} (u \cdot \partial_{\mathbf{n}} v - v \partial_{\mathbf{n}} u) ds,
$$
where $\partial_{\mathbf{n}}$ refers to differentiation with respect to the normal vector and $ds$ denotes integration with respect to arc length.
The choice $$\Omega = \square_I^n, \quad u = y, \quad v = |b(z)-b(z_I)|^2$$ yields
\begin{equation*}
\int_{\square_I^n} y \cdot  |2b'(z)|^2 \, dxdy  =  \int_{\square_I^n} y \cdot \Delta |b(z) - b(z_I)|^2 \, dxdy, \qquad \qquad \quad
\end{equation*}
\begin{equation*}
\quad \qquad \qquad =  - \int_{\partial \square_I^n} \partial_{\mathbf{n}} y \cdot |b(z) - b_I(z)|^2 \, ds + \int_{\partial \square_I^n} y \cdot  \partial_{\mathbf{n}} |b(z)-b(z_I)|^2 \, ds.
\end{equation*}
 It is evident that
\begin{equation}
- \int_{\partial \square_I^n} \partial_{\mathbf{n}}  y \cdot |b(z)-b(z_I)|^2 \, ds =  \int_{\bottom( \square_I^n)} |b(z)-b(z_I)|^2 \, ds - 
\mathcal O(1).
\end{equation}
For the error term, 
$$
\int_{\partial \square_I^n} y \cdot \partial_{\mathbf{n}} |b(z)-b(z_I)|^2 \, ds \, \lesssim \,  \int_{\partial \square_I^n} |b(z)-b(z_I)| \, ds.
$$
From the definition of a Bloch function (\ref{eq:def-bloch}), the integral of $|b(z)-b(z_I)|$ over the top, left and right sides of $\partial \square_I^n$ is $\mathcal O(1)$, while according to 
(\ref{eq:l1norm}), the integral over the bottom side is $\mathcal O(\sqrt{n})$.
Summarizing, we see that
\begin{equation}
\fint_{\square^n} \biggl |\frac{2b'(z)}{\rho_{\mathbb{H}}} \biggr |^2 \, \frac{|dz|^2}{y} = \frac{1}{n \log 2} \int_{\bottom( \square^n)} |b(z)-b(z_I)|^2 \, ds +
 \mathcal O (1/\sqrt n).
\end{equation}
Combining with (\ref{eq:sqf-appx}) gives another error of $\mathcal O (1/\sqrt n)$ and proves (\ref{eq:ltg1}). For (\ref{eq:ltg2}), it suffices to 
repeat the argument with $u = \re b$ (in place of $b$) and use $|b'|^2 = 2|\nabla u|^2$, while (\ref{eq:ltg3}) follows from polarization.
 \end{proof}

\section{Applications to Bloch functions} 
 Let  $\bH_k$ denote the class  
of conformal maps $f: \mathbb{H} \to \mathbb{C}$ which admit a $k$-quasiconformal extension
to the plane and fix the points $0,1,\infty$. As discussed in the previous section, for $f \in \mathbf{H}_k$, the associated Bloch function
 $b_f = \log f' \in \mathcal B(\mathbb{H})$ defines a dyadic martingale $B$ on $[0,1]$. We define the
 asymptotic variance, LIL constant and integral means of $b$ as
 $$
 \frac{1}{\log 2} \cdot \sigma^2(B),  \quad \frac{1}{\log 2} \cdot C_{\LIL}^2(B),  \quad \frac{1}{\log 2} \cdot \beta_B(t),
 $$
 respectively. The factor $\chi = \log 2$ comes from the height of the boxes in the dyadic grid (as measured in the hyperbolic metric). It plays the role of the  Lyapunov exponent, cf. \cite[Theorem 2.7]{mcmullen}. If one instead works with the $p$-adic grid, then the normalizing factor would be $\chi = \log p$.
 
In purely function-theoretic terms, the 
 asymptotic variance of a Bloch function $b \in \mathcal B(\mathbb{H})$ is given by
\begin{align}
 \sigma^2_{[0,1]}(b) & = \limsup_{y \to 0^+} \, \frac{1}{|\log y|}  \int_0^1 |b(x+iy)|^2 dx, \\ 
 \label{eq:avar-H}
& = \limsup_{h \to 0^+} \, \frac{1}{|\log h|} \int_h^1 \int_0^1 \biggl |\frac{2b'(x+iy)}{\rho_{\mathbb{H}}}\biggr |^2 \, \frac{|dz|^2}{y}.
\end{align}
More generally, in \cite[Section 6]{mcmullen}, McMullen  showed that one can compute the asymptotic variance by examining C\'esaro averages of integral means that
involve
higher order derivatives. 

 It is not difficult to show that the expressions
 $$
 \sup_{f \in \bH_k} \sigma^2(b_f), \quad  \sup_{f \in \bH_k} C_{\LIL}^2(b_f), \quad  \sup_{f \in \bH_k} \beta_{b_f}(t),
 $$
coincide with their analogues for the class $\bS_k$ from the introduction.

 \begin{theorem}[Complexification]
For any Bloch function $b \in \mathcal B(\mathbb{H})$,
 \label{lil-circular}
$$
\sigma^2(\re b) = (1/2) \cdot \sigma^2(b), \qquad C_{\LIL}(\re b) = C_{\LIL}(b).
$$
\end{theorem}

\begin{proof}
(i) The first statement follows from (\ref{eq:ltg2})  and the definitions
\begin{equation}
\label{eq:complexification1}
 \frac{1}{n} \int_0^1 \bigl | \re B_n(x) - \re B_{[0,1]} \bigr | ^2 \, dx = \var^n_{[0,1]}[\re B] 
\end{equation}
and 
\begin{equation}
\label{eq:complexification2}
 \frac{1}{n} \int_0^1 | B_n(x) - B_{[0,1]} |^2 \, dx = \var^n_{[0,1]}[B]. 
\end{equation}

(ii) 
For the second statement, note that the function $\theta \to C_{\LIL}(\re e^{i\theta} B)$ is continuous and
$$
C_{\LIL}(B) =  \sup_{\theta \in [0,2\pi)} C_{\LIL}(\re e^{i\theta} B).
$$
We  must therefore show $C_{\LIL}(\re e^{i\theta} B) \le C_{\LIL}(\re B)$, for any $\theta \in [0,2\pi)$. However, if view $\re B$ and $\re e^{i\theta} B$ as
 $p$-adic martingales with $p = 2^n$
large, then by (\ref{eq:ltg2}), their local variances are approximately equal. The assertion now follows from 
Lemma \ref{same-variance}.
\end{proof}
\begin{remark}
For lacunary series, the equality $C_{\LIL}(\re b) = C_{\LIL}(b)$ goes back to the 1959 work of M.~Weiss \cite{weiss}.
\end{remark}

\section{The Box Lemma}

The proofs of Theorems \ref{dynamical-connections} and \ref{main-thm} are now completed by the {\em Box Lemma} from \cite{qcdim} which describes the average non-linearity $n_f := f''/f' = (\log f')'$ of conformal mappings:

\begin{lemma} 
\label{boxcart-global}
{\em (i)} Fix $0 < k < 1$.
Given $\varepsilon > 0$, there exists $n \ge 1$ sufficiently large so that for any $n$-box $\square_I^n \subset \mathbb{H}$
and any conformal map $f \in \mathbf{H}_k$,
 \begin{equation}
 \label{eq:boxcart2}
 \fint_{\square_I^n} \biggl |\frac{2n_f}{\rho_{\mathbb{H}}}(z) \biggr |^2 \, \frac{|dz|^2}{y} \, < \, \Sigma^2(k) + \varepsilon.
 \end{equation}
 
 {\em (ii)} Conversely, for any $\varepsilon > 0$, there exists a conformal map $f \in \mathbf{H}_k$, whose dilatation $\bel f := \overline{\partial} f/\partial f$ is periodic with respect
to the $2^n$-adic grid for some $n \ge 1$,
and which satisfies
  \begin{equation}
 \label{eq:boxcart3}
 \fint_{\square_I^n} \biggl |\frac{2n_f}{\rho_{\mathbb{H}}}(z) \biggr |^2 \, \frac{|dz|^2}{y} \, > \, \Sigma^2(k) - \varepsilon,
 \end{equation}
  on every $n$-box $\square_I^n$.
\end{lemma}

\begin{remark}
The proof given in \cite{qcdim} forces us to restrict our attention to classes of conformal maps with bounded distortion. It would be interesting to know if 
a variant of the box lemma holds
for all conformal maps with $\lim_{k\to 1^-} \Sigma^2(k)$ in place of $\Sigma^2(k)$.
\end{remark}

In view of Theorems \ref{local-variance} and \ref{variance-thm}, 
(i) gives the upper bound in Theorem \ref{main-thm}, while (ii) gives the lower bound.
The notion of periodic Beltrami coefficients will be discussed below in Section \ref{sec:pbc}.
In order to state the infinitesimal version of the box lemma, note that the formula for the Beurling transform
(\ref{eq:beurling-def}) may not converge if $\mu$ is not compactly supported. Therefore, we are obliged to work with a modified Beurling transform
\begin{equation}
\label{eq:beurling22}
 \mathcal S^\# \mu(z) \, =\, -\frac{1}{\pi} \int_{\mathbb{H}} \mu(\zeta) \biggl [ \frac{1}{(\zeta-z)^2} - \frac{1}{\zeta^2} \biggr ] \, |d\zeta|^2.
 \end{equation} 
However, the formula for the derivative remains the same:
\begin{equation}
\label{eq:beurling23}
 \text{``}(\mathcal  S \mu)'(z)\text{''} \, := \,  (\mathcal S^\# \mu)'(z)  \, =\, -\frac{2}{\pi} \int_{\mathbb{H}} \frac{\mu(\zeta)}{(\zeta-z)^3} \, |d\zeta|^2.
 \end{equation}

In \cite{qcdim}, the infinitesimal analogue of the box lemma was proved with a quantitative relation between the box size and the error term:
 
\begin{lemma}
\label{boxcart}
{\em (i)} 
For any Beltrami coefficient $\mu$ with $|\mu| \le \chi_\mathbb{\overline{H}}$ and $n$-box $\square_I^n \subset \mathbb{H}$,
 \begin{equation}
 \label{eq:boxcart}
 \fint_{\square_I^n} \, \biggl |\frac{2(\mathcal S\mu)'}{\rho_{\mathbb{H}}}(z) \biggr |^2 \, \frac{|dz|^2}{y} \, < \, \Sigma^2 + C/n.
 \end{equation}
 
 {\em (ii)} Conversely, for $n \ge 1$, there exists a Beltrami coefficient $\mu$, periodic with respect to the $2^n$-adic grid, which satisfies
 \begin{equation}
 \label{eq:boxcart-local-converse}
 \fint_{\square_I^n} \, \biggl |\frac{2(\mathcal S\mu)'}{\rho_{\mathbb{H}}}(z) \biggr |^2 \, \frac{|dz|^2}{y}  \, > \, \Sigma^2 - C/n
 \end{equation}
 on every $n$-box $\square_I^n$.
\end{lemma}

The quantitative estimate will be exploited in Section \ref{sec:ims}.

\subsection{Periodic Beltrami coefficients}
\label{sec:pbc}
Given two intervals $I, J \subset \mathbb{R}$, let $L_{I, J}(z) = Az + B$ be the unique linear map with $A > 0, B \in \mathbb{R}$ that maps $I$ to $J$.
 For a box  $\square$, we denote its reflection in the real line by $\overline{\square}$.
Suppose $\mu$ is a Beltrami coefficient supported on the lower half-plane.
We say that $\mu$ is {\em periodic} (with respect to the dyadic grid) if for any two dyadic intervals
$I, J \subset \mathbb{R}$ with $|I|, |J| \le 1$,
$
\mu|_{\overline{\square}_I} = L_{I, J}^*(\mu|_{\overline{\square}_J}).
$
We typically assume that
$\mu$ is supported on the strip $$\{w : -1 < \im w < 0\},$$ in order for $\mu$ to be invariant under translation by 1. In this case, $\mu$ descends
to a Beltrami coefficient on the disk via the exponential mapping, which is eventually-invariant under $z \to z^2$.
The notion of a Beltrami coefficient periodic with respect to the $p$-adic grid is defined similarly.

Before continuing further, we define a {\em dyadic box} in the unit disk to be the image of $\square_I$, $|I| \le 1$, under the exponential mapping $\xi(w) = \exp(2\pi i w).$
Reflecting in the unit circle, we obtain a dyadic box in the exterior unit disk. Note that these boxes are not geometric rectangles, nor do they tile
$\mathbb{D}$ or $\mathbb{D}^*$ completely.

\subsection{Dynamical Beltrami coefficients}
We now consider two classes of dynamical Beltrami coefficients on the unit disk that naturally arise in complex dynamics and Teich\"muller theory:
\begin{itemize}
\item $M_{\bl} = \bigcup_f M_f(\mathbb{D})$ consists of Beltrami coefficients that are {\em eventually-invariant} under some finite
\emph{Blaschke product} $f(z) = z \prod_{i=1}^{d-1} \frac{z-a_i}{1-\overline{a_i}z}$, i.e.~Beltrami coefficients which satisfy $f^*\mu = \mu$ in some open neighbourhood of the unit circle.
\item $M_{\fuchs} = \bigcup_\Gamma M_{\Gamma}(\mathbb{D})$ consists of Beltrami coefficients  that are invariant under some co-compact Fuchsian group $\Gamma$,
i.e.~$\gamma^*\mu = \mu$ for all $\gamma \in \Gamma$.
\end{itemize}

Suppose $\mu$ belongs to one of the two classes of Beltrami coefficients above, with $\|\mu\|_\infty < 1$. 
We view $f = w^{\mu}$ as a conformal map of the exterior unit disk. 
From the construction, the image of the unit circle $f(\mathbb{S}^1)$ is a Julia set or a limit set of a quasi-Fuchsian group. 
Using the ergodicity of the geodesic flow on the unit tangent bundle $T_1 X$ (Fuchsian case) or Riemann surface lamination $\hat X_B$ (Blaschke case), it is not hard to show that for any $\varepsilon > 0$, there exists $n_0$ sufficiently large,
\begin{equation}
\label{eq:ergodicity}
\sigma^2(\log f')  - \varepsilon \, < \,  \fint_{\square_I^n} \, \biggl |\frac{2n_f}{\rho_*}(z) \biggr |^2 \,  \rho_*|dz|^2 \, < \,
\sigma^2(\log f') + \varepsilon
\end{equation}
for any $n$-box $\square^n_I \subset \mathbb{D}^*$ with $n \ge n_0$.
Applying Theorems \ref{local-variance} and \ref{variance-thm} 
shows that Theorem \ref{dynamical-connections} holds for conformal maps $f = w^{\mu}$ with $\mu \in M_{\bl}$ or $M_{\fuchs}$.
More generally, one can prove (\ref{eq:ergodicity}) for conformal maps to simply-connected domains bounded by Jordan repellers, see \cite[Section 8]{AIPP} for a definition.
The reader interested in working out the details can consult \cite{mcmullen}.

\section{Applications to integral means}
\label{sec:ims}

In this section, we use martingale techniques to study the rescaled integral means spectrum (\ref{eq:inf-ims}).
For a fixed $\tau \in \mathbb{C}$, 
the (uniform) convergence $$\beta_{\log (w^{k\mu})'}(\tau/k) \, \to \, \beta_{\mathcal S\mu}(\tau), \qquad k \to 0,$$ can be justified
using (\ref{eq:quadratic-error}) and a variant of Lemma \ref{same-variance} for integral means, so the rescaled integral means spectrum is well-defined.
Moving to the upper half-plane, we are led to analyze the asymptotic expansion of 
$$
B_0(\tau) \, = \,\lim_{k \to 0}  B_k(\tau/k)  \, =\, \sup_{|\mu| \le \chi_{\overline{\mathbb{H}}}} \beta_{\mathcal S^\#\mu} (\tau),
$$
near $\tau = 0$.
Let $B$ be the dyadic martingale associated to the Bloch function $\mathcal S^\#\mu$, $|\mu| \le \chi_{\overline{\mathbb{H}}}$ and $S = \re B$ be its real part. It suffices to estimate $\beta_S(t)$
with $t \in \mathbb{R}$.
We view $S$ as a $p$-adic martingale with $p = 2^n$, where the parameter $n$ will be chosen momentarily.
 Suppose $I$ is a $2^n$-adic interval and
  $I_1, I_2, \dots, I_{2^n}$ are its $2^n$-adic children. Then,
\begin{align*}
\frac{1}{2^n} \sum_{j=1}^{2^n} \exp(t \Delta_j) & = 
1 + \frac{t^2}{2} \biggl ( \frac{1}{2^n} \sum_{j=1}^{2^n} \Delta_j^2 \biggr) + \sum_{k \ge 3}  \frac{t^k}{k!} 
\biggl ( \frac{1}{2^n} \sum_{j=1}^{2^n} \Delta_j^k \biggr ). \\
& =
 1 + \frac{n t^2}{2} \var_I^n + \mathcal O \biggl ( \sum_{k \ge 3}  \frac{t^k}{k!} \cdot (Cn)^{k/2} \biggr ).
\end{align*}
Above, we used (\ref{eq:moments}) to estimate the remainder term.
If $b = \mathcal S^\#\mu$, $|\mu| \le \chi_{\overline{\mathbb H}}$, from Lemma \ref{boxcart}(i), we see that the above expression is bounded by
$$
\le 1 + \frac{n t^2}{2} \biggl (\frac{\Sigma^2 \log 2}{2} + \mathcal O(1/\sqrt{n}) + \mathcal O(tn^{1/2}) + \dots \biggr ).
$$
Note that in order to use martingale techniques, we had to downgrade the box estimate with $\Sigma^2/2 + C/n$ to the variance bound $\var_I^n/\log 2 \le \Sigma^2/2 + C/\sqrt{n}$, cf.~Theorem \ref{variance-thm}. Hence,
$$
\frac{1}{n \log 2} \log \biggl [ \frac{1}{2^n} \sum_{j=1}^{2^n} \exp(t \Delta_j)\biggr ] \le \frac{t^2}{2} \biggl (\frac{\Sigma^2}{2} + \mathcal O(1/\sqrt{n}) + \mathcal O(tn^{1/2}) + \dots \biggr ).
$$
Taking $n = \lfloor t^{-1} \rfloor$ leads to the estimate
\begin{equation}
\label{eq:b0expansion}
B_0(t) \le 1+\Sigma^2t^2/4+\mathcal O(|t|^{5/2}).
 \end{equation}
 By using Lemma \ref{boxcart}(ii), the above reasoning gives a lower bound for integral means  which shows that (\ref{eq:b0expansion}) is an equality.
As mentioned in the introduction, if one avoids martingales, one can obtain a better remainder term than $\mathcal O(|t|^{5/2})$.

\section{A central limit theorem}
\label{sec:clt}

Suppose $b = P\mu$ with $|\mu| \le \chi_{\mathbb{D}}$. Consider the function 
\begin{equation}
\label{eq:rescaled-bloch2}
 \tilde b_r(\theta) := \frac{b(re^{i \theta})}{\sqrt {|\log(1-r)|}}.
\end{equation}
The sub-Gaussian estimate (\ref{eq:subgaussian})  shows that most of the integral
$\int (\re \tilde{b}_r(\theta))^2 d\theta$ comes from the set
$$
A_\delta = \bigl\{\theta: -1/\delta < \re \tilde{b}_r(\theta) < 1/\delta\bigr\},
$$
that is, by making $\delta$ small, we can guarantee that
$$
\int_{A_\delta^c} (\re \tilde{b}_r(\theta))^2 d\theta < \varepsilon,
$$
where the estimate is uniform over all functions $b$ of the form above.
We now show that if $\sigma^2(b, r)$ is close to $\Sigma^2$, then the distribution of $\re \tilde{b}_r(\theta)$ is close to a Gaussian of mean 0
and variance $\Sigma^2/2$.

\begin{theorem}
\label{real-gaussian}
Suppose $|\mu| \le \chi_{\mathbb{D}}$.
Given $\varepsilon > 0$, there exists a $\delta > 0$ such that if $r$ is sufficiently close to 1 and
 $$\sigma^2(P\mu,r) \, = \, \frac{1}{2\pi |\log(1-r)|} \int_{|z|=r} |P\mu(z)|^2 \, |dz| \, > \, \Sigma^2 - \delta,$$ 
then for any $t \in \mathbb{R}$,
$$
\bigl | \mathbb{P}(\re \tilde{b}_r(\theta) < t) - \mathcal N_{0, \Sigma^2/2}(t) \bigr | < \varepsilon.
$$
The same statement holds with $\im \tilde{b}_r(\theta)$ as well.
\end{theorem}

\subsection{Characteristic functions}

Converting to the upper half-plane, let $B$ be the $p$-adic martingale associated to the Bloch function $b = \mathcal S^\#\mu$, $|\mu| \le \chi_{\overline{\mathbb{H}}}$ and $S = \re B$ be its real part. 
Here, we choose $p$ sufficiently large to guarantee that the box averages  (\ref{eq:boxcart}) are at most
$\Sigma^2+\delta_1$. Let $\chi = \log p$ be the ``Lyapunov exponent'' of the $p$-adic grid.
As is standard \cite{durrett}, to prove the central limit theorem, one must examine the characteristic functions 
$$
\varphi_n(t) = \mathbb{E} \exp \biggl ( i \cdot \frac{t S_n}{\sqrt{n\chi}}\biggr).
$$
However, since $\Delta_j(x)$ may not be constant in $x$, martingale jumps are usually not independent. Instead, we leverage the fact that
 the local variance is approximately constant.
Observe that if $I$ is a $p$-adic interval and $I_1, I_2, \dots, I_{p}$ are its children, then for $t$ small,
$$
\frac{1}{p} \sum_{j=1}^{p} e^{it \Delta_j} = 1 - \frac{t^2}{2} \var_I + \mathcal O(t^3).
$$
If $S$ had {\em constant local variance}, that is if $\var_I = \sigma^2$ for all $I$, then the characteristic function
of $S_n/\sqrt{n\chi}$ would be simply
$$
\varphi_n = \biggl(1 - \frac{\sigma^2}{\chi} \cdot \frac{ t^2}{2n} + \mathcal O(t^3) \biggr)^n.
$$
Taking $n \to \infty$, one obtains
$$
\varphi \,=\, \lim_{n \to \infty} \varphi_n \, =\, \exp \biggl (-\frac{\sigma^2}{\chi} \cdot \frac{t^2}{2} \biggr ),
$$
which is the characteristic function of the Gaussian $\mathcal N_{0,\sigma^2/\chi}$. 

For the problem at hand, we must slightly relax the assumption of constant local variance. 
First, note that if the local variance is pinched
\begin{equation}
\label{eq:goodbox}
\Sigma^2/2 - \delta_2 \, \le \, \var_I/\chi \, \le \, \Sigma^2/2 + \delta_2,
\end{equation} 
then the characteristic functions $\varphi_n$ satisfy
\begin{equation}
\label{eq:lightbulb}
\ \varphi_n(t) = \exp \Bigl (-\sigma_n \cdot  t^2/2+ o(t^2) \Bigr ), \quad \text{with } |\sigma_n - \Sigma^2/2| \le \delta_2.
\end{equation}
In this case, the inversion formula for characteristic functions guarantees that for any $n \ge 1$, the distribution of
$S_n/\sqrt{n\chi}$ is close to  $\mathcal N_{0,\Sigma^2/2}$.

\subsection{Allowing bad boxes}
Additionally, we must allow a small proportion of $p$-adic intervals $I$ to be bad where we only have weak control on $\var_I$ coming from the bounded increments assumption -- note that the Bloch norm $\| b\|_{\mathcal B(\mathbb{H})}$ is bounded by a universal constant for $b = \mathcal S^\#\mu$, $|\mu| \le \chi_{\overline{\mathbb{H}}}$.
For a $p$-adic interval, write
$$
\square_I^{(p)} \, = \, \Bigl \{ w : \re w \in I, \, \im w \in \bigl [p^{-1}|I|, \,|I| \bigr ] \Bigr \}.
$$
Set
$$
\square^n \, = \, \bigl \{ w : \re w \in [0,1], \, \im w \in [p^{-n}, \, 1] \bigr  \}.
$$
Fix $n \ge 1$. Call a $p$-adic box $\square_I^{(p)}$ with  $|I| \ge p^{-n+1}$ {\em good} if
(\ref{eq:goodbox}) holds and 
 {\em bad} otherwise.
Let $\mathscr E \subset \square^n$ denote the union of bad boxes.
Inspecting (\ref{eq:avar-H}), we see that if 
\begin{equation}
\sigma^2(S, n)/\chi \, := \, \frac{1}{n \chi} \int_0^1 S_n^2 \, dx \, \ge \, \Sigma^2/2-\delta,
\end{equation}
 then
 \begin{equation}
\label{eq:badbox}
\fint_{\square^n} \chi_{\mathscr E} \cdot \frac{|dz|^2}{y} \, \le \, \delta_3.
\end{equation}
Therefore, to prove Theorem \ref{real-gaussian}, it suffices to show:

\begin{lemma}
\label{badbox-lemma}
If the parameters $\{ \delta_i\}$ above are sufficiently small, then $$\bigl | \mathbb P(S_n/\sqrt{n\chi} < t) - \mathcal N_{0,\Sigma^2/2}(t) \bigr | \le \varepsilon.$$
\end{lemma}
 
\begin{proof}
Write $S = S^{\good} + S^{\bad}$ as a sum of two martingales, where the local variance
of $S^{\good}$ is close to $\Sigma^2$ on all intervals, while the increments of $S^{\bad}$ are non-zero only on bad intervals. 
We may form $S^{\good}$ from $S$ by adjusting the jumps on the bad intervals, and defining $S^{\bad} := S - S^{\good}$ to be the difference.
From the construction, it is clear that 
$$
\sigma^2(S^{\bad}, n) \, = \, \frac{1}{n} \int_0^1 (S^{\bad}_n)^2 \, dx 
\, \lesssim \, \delta_3,
$$ 
which shows that $S^{\bad}_n/\sqrt{n\chi}$ is small outside of a set of small measure. Therefore, the distribution of $S_n/\sqrt{n\chi}$ is roughly that of $S^{\good}_n/\sqrt{n\chi}$, which we
already know to be approximately Gaussian.
\end{proof}

The proof of Theorem \ref{complex-gaussian} is similar except one considers characteristic functions of two variables
$$
\varphi_n(s, t) =  \mathbb{E} \exp \biggl ( i \cdot \frac{s \re B + t \im B}{\sqrt{n\chi}}\biggr).
$$
and uses the approximate orthogonality  (\ref{eq:ltg3}) between $\re B$ and $\im B$ to show 
$\varphi_n(s, t) \approx \exp \Bigl (- \frac{\Sigma^2(s^2+t^2)}{4} \Bigr )$. We leave the details to the reader.

\bibliographystyle{amsplain}

\end{document}